\documentclass[10pt,a4,twoside]{amsart}
\usepackage{charter, amsmath,amsbsy,amsfonts,amssymb,
stmaryrd,amsthm,mathrsfs,graphicx, amscd, tikz-cd, cancel}
\usepackage{srcltx}
\usepackage[mathscr]{eucal}

\usepackage[
	hypertexnames=false,
	hyperindex,
	pagebackref,
	pdftex,
	breaklinks=true,
	bookmarks=false,
	colorlinks,
	linkcolor=blue,
	citecolor=red,
	urlcolor=red,
]{hyperref}

\def\CC{{\mathbb C}}

\def\PP{{\mathbb P}}
\def\QQ{{\mathbb Q}} 
\def\RR{{\mathbb R}}

\def\ZZ{{\mathbb Z}}

\newcommand\ssm{\smallsetminus}

\def\Acal{{\mathcal A}}

\def\Fcal{{\mathcal F}}

\def\Ical{{\mathcal I}}

\def\Tcal{{\mathcal T}}

\def\bb{{\mathit bb}}
\def\pr{{\rm pr}}

\def\la{\langle}
\def\ra{\rangle}

\def\bs{\backslash}
\def\half{\frac{1}{2}}
\def\pt{{\scriptscriptstyle\bullet}}

\newcommand\lognrm{\operatorname{lgnm}}
\newcommand\GL{\operatorname{GL}}

\newcommand\SL{\operatorname{SL}}

\newcommand\Sp{\operatorname{Sp}}

\newcommand\Spcal{\operatorname{{\mathcal S}\!{\mathit p}}}
\newcommand\SLcal{\operatorname{{\mathcal S}\!{\mathcal L}}}

\newcommand\sym{\operatorname{Sym}}
\newcommand\Hom{\operatorname{Hom}}

\newtheorem{theorem}{Theorem}[section] 
\newtheorem{lemma}[theorem]{Lemma}

\newtheorem{corollary}[theorem]{Corollary}

\theoremstyle{definition}

\theoremstyle{remark} \newtheorem{remark}[theorem]{Remark}

\begin{document}

\def\subjclassname{\textup{2010} Mathematics Subject Classification}
\subjclass[2010]{14G35, 32S35}

\title{The stable cohomology of the Satake compactification of $\Acal_g$}

\author{Jiaming Chen}
\address{Universit\'e Pierre et Marie Curie, 4 Place Jussieu, 75005 Paris} 
\email{jiaming.chenl@gmail.com}
\author{Eduard Looijenga}
\address{Yau Mathematical Sciences Center Tsinghua University Beijing (China)}
\email{eduard@math.tsinghua.edu.cn}

\begin{abstract}
Charney and Lee \cite{charneylee:satake}  have shown that the rational cohomology of the Satake-Baily-Borel  compactification $\Acal_g^\bb$ of $\Acal_g$ stabilizes as $g\to\infty$ and they computed this stable cohomology as a Hopf algebra.  We give a relatively simple  algebro-geometric proof of their theorem and show that this stable cohomology comes with a mixed Hodge structure of which we determine the Hodge numbers. 
We find that the mixed Hodge structure on the primitive cohomology in degrees $4r+2$ with $r\ge 1$ is an  extension of $\QQ (-2r-1)$ by $\QQ(0)$; in particular, it is not pure.
\end{abstract}

\maketitle
\section{The theorem} 

Let $\Acal_g=\Acal_g(\CC)$ denote the coarse moduli space of principally polarized complex abelian varieties of genus $g$ endowed  with the analytic (Hausdorff) topology. Recall that the  Satake-Baily-Borel compactification $j_g:\Acal_g\subset \Acal_g^\bb$  realizes $\Acal_g$ as a Zariski open-dense subset in a normal projective variety $ \Acal_g^\bb$. Forming the product of two principally polarized abelian varieties defines a morphism of moduli spaces $\Acal_g\times \Acal_{g'}\to \Acal_{g+g'}$  which  extends to these compactifications: we have a commutative diagram
\begin{equation}\label{eqn:monoid}
\begin{CD}
\Acal_g\times \Acal_{g'} @>>> \Acal_{g+g'}\\
@V{j_g\times j_{g'}}VV  @VV{j_{g+g'}}V\\
\Acal^\bb_g\times \Acal^\bb_{g'} @>>> \Acal^\bb_{g+g'}
\end{CD}
\end{equation}
By taking $g'=1$ and choosing a point of $\Acal_1$, we get the  `stabilization maps'
\begin{equation}\label{eqn:stable}
\begin{CD}
\Acal_g@>>> \Acal_{g+1}\\
@V{j_g}VV  @VV{j_{g+1}}V\\
\Acal^\bb_g @>>> \Acal^\bb_{g+1}
\end{CD}
\end{equation}
whose homotopy type does not depend on the point we choose, for $\Acal_1$ is isomorphic to the affine line and hence connected. Since we are only concerned with homotopy classes and commutativity up to homotopy, we can for 
the definition of the map $\Acal^\bb_g\to\Acal^\bb_{g+1}$ even choose  this point to be represented by the singleton $\Acal_0$. Then
this map is a homeomorphism onto the Satake boundary (since $\Acal^\bb_1\cong\PP^1$ the maps are not just homotopic, but even induce the same map on Chow groups). We shall see that this gives rise to two Hopf algebras with a mixed Hodge structure.

Before we proceed, let us recall that $\Acal_g$ is a locally symmetric variety associated to the $\QQ$-algebraic group $\Spcal_g$ and that the $\QQ$-rank of $\Spcal_g$  is $g$. According to Borel and Serre (Cor.\ 11.4.3 of \cite{borel-serre:corners})  the virtual cohomological dimension of $\Sp (2g, \ZZ)$ equals  $\dim_\RR\Acal_g -g$. This implies that the rational cohomology of $\Acal_g$, and more generally, the cohomology of a sheaf $\Fcal$ on $\Acal_g$  defined by a representation of $\Sp (2g, \ZZ)$ on a $\QQ$-vector space vanishes in degrees  $>\dim_\RR\Acal_g -g$. Since $\Acal_g$ is an orbifold, this is via Poincar\'e-Lefschetz  equivalent  to $H^k_c(\Acal_g;\Fcal)$ being zero for $k<g$. We shall use this basic fact in the proofs of \ref {lemma:stable0} below and  of \ref{lemma:stable1}.

\begin{lemma}\label{lemma:stable0}
The stabilization maps  $\Acal_g\hookrightarrow \Acal_{g+1}$  (multiplication by a fixed elliptic curve) and 
$\Acal^\bb_g\to\Acal^\bb_{g+1}$ (mapping onto the boundary) defined above induce on rational cohomology  an isomorphism in degree $<g$ and are injective in degree $g$.
\end{lemma}
\begin{proof}
Recall that $\Acal_g$ is a locally symmetric variety associated to the $\QQ$-algebraic group $\Spcal_g$ and that the $\QQ$-rank of $\Spcal_g$  is $g$. The first assertion then follows from a theorem of Borel \cite{borel:stable1}.
The second stability assertion is equivalent to the vanishing of the relative cohomology 
$H^k(\Acal_{g+1}^\bb,\Acal_{g}^\bb; \QQ)$  for $k\le g$.
As this is just  $H^k_c(\Acal_{g+1}; \QQ)$, this follows from the Borel-Serre result quoted above. 
\end{proof}

We  then form the stable rational cohomology spaces  
\[
H^k(\Acal_\infty; \QQ):=\varinjlim_g H^k(\Acal_g; \QQ), \quad H^k(\Acal^\bb_\infty; \QQ):=\varinjlim_g H^k(\Acal^\bb_g; \QQ),
\]
where the notation is only suggestive, for there is here no pretense of introducing spaces $\Acal_\infty$ and $\Acal^\bb_\infty$.
If we take the direct sum over $k$  we get a  $\QQ$-algebra in either case.
It follows from the homotopy commutativity of the diagram (\ref{eqn:stable}) above that the  inclusions 
$j_g$ define a graded $\QQ$-algebra homomorphism 
\[
j_\infty^*: H^\pt(\Acal^\bb_\infty; \QQ)\to H^\pt(\Acal_\infty; \QQ).
\] 
The multiplication  maps exhibited in diagram  (\ref{eqn:monoid}) are (almost by definition) compatible with the stabilization maps  and hence induce
a graded coproduct on either algebra so that 
$j_\infty^*$ becomes a homomorphism of (graded bicommutative) Hopf algebras. Since the multiplication maps and the stability 
maps are morphisms in the category of complex algebraic varieties, these Hopf algebras come with a natural mixed Hodge structure such that 
$j_\infty^*$ is also a morphism in the mixed Hodge category. The Hopf algebra $H^\pt(\Acal_\infty; \QQ)$ is well-known and due to Borel 
\cite{borel:stable1}: it has as its primitive elements classes $ch_{2r+1}\in H^{4r+2}(\Acal_\infty; \QQ)$, $r\ge 0$, where $ch_{2r+1}$ restricts to  
$\Acal_g$ as the rational $(2r+1)$th Chern character of the Hodge bundle on $\Acal_g$,  and so 
$H^\pt(\Acal_\infty; \QQ)=\QQ[ch_1, ch_3, ch_5, \dots ]$ with $ch_{2r+1}$ of type $(2r+1,2r+1)$ (if we are happy with multiplicative generators, we can just as well replace  $ch_{2r+1}$ by the corresponding Chern class $c_{2r+1}$, for  $c_{2r+1}$ is expressed universally in
$[ch_1, ch_3, ch_5,\dots ch_{2r+1}$ and vice versa). The principal and essentially only result of this paper is Theorem \ref{thm:main} below. Its first assertion is due Charney and Lee \cite{charneylee:satake}, who derive this from a determination of a limit of homotopy types. We shall obtain this in a relatively  elementary manner  by means of algebraic geometry and the  classical vanishing results of Borel and Borel-Serre. Our approach has the advantage that it helps us to understand the new classes that appear here geometrically,  to the extent that this enables us to determine their Hodge type. We address the homotopy discussion of Charney and Lee and a  generalization thereof in another paper \cite{chen-looijenga} that will not be used here.

\begin{theorem}\label{thm:main}
The graded Hopf algebra  $H^\pt(\Acal^\bb_\infty; \QQ)$ has for every integer $r\ge 1$ a primitive generator $y_r$  of degree $4r+2$ 
and for every integer  $r\ge 0$ a primitive generator $\widetilde{ch}_{2r+1}$ of degree $4r+2$ such that the map
$j_\infty^*: H^\pt(\Acal^\bb_\infty; \QQ)\to H^\pt(\Acal_\infty; \QQ)$ sends $\widetilde{ch}_{2r+1}$ to $ch_{2r+1}$ and is zero on $y_r$ when $r\ge 1$. In particular, if $\tilde c_{2r+1}\in H^{4r+2}(\Acal^\bb_\infty; \QQ)$ denotes the lift of $c_{2r+1}\in H^{4r+2}(\Acal_\infty; \QQ)$ that is obtained from our choice of the $\widetilde{ch}_{1}, \dots ,\widetilde{ch}_{2r+1}$ (as a universal polynomial in these classes), then $H^\pt(\Acal^\bb_\infty; \QQ)=\QQ[y_1, y_2,y_3\dots, \tilde c_1, \tilde c_3,\tilde c_5,\dots ]$ as a commutative $\QQ$-algebra.

The mixed Hodge structure on $H^\pt(\Acal^\bb_\infty; \QQ)$ is such that $y_r$ is of bidegree $(0,0)$ and 
$\widetilde{ch}_{2r+1}$ (or equivalently, $\tilde c_{2r+1}$) is of bidegree $(2r+1,2r+1)$.
\end{theorem}

\begin{remark}\label{rem:tate1}
So for $r\ge 1$, the primitive part $H^{4r+2}_\pr(\Acal_\infty^\bb;\QQ)$ of the  Hopf algebra  $H^\pt(\Acal^\bb_\infty; \QQ)$ is two-dimensional in degree $4r+2$ and defines 
a Tate extension 
\[
0\to\QQ\to H^{4r+2}_\pr(\Acal_\infty^\bb;\QQ)\to \QQ(-2r-1)\to 0,
\]
with $\QQ$ spanned by $y_r$ and $\QQ(-2r-1)$ spanned by $ch_{2r+1}$. We discuss the nature of this extension briefly in Remark
\ref{rem:tate2} below.
\end{remark}

We thank the referee for helpful comments on an earlier version. These led to an improved  exposition.

\section{Determination of the stable cohomology as a Hopf algebra}
According to \cite{faltingschai} Ch.\ V, Thm.\ 2.3 part (3), $\Acal^\bb_g\ssm \Acal_g$ is as a variety isomorphic to $\Acal^\bb_{g-1}$. In particular,  we have a partition into locally closed subvarieties: $\Acal^\bb_g=\Acal_g\sqcup \Acal_{g-1}\sqcup\cdots \sqcup\Acal_0$.

We will use the fact that the higher direct images  
$R^\pt j_{g*}\QQ_{\Acal_g}$ are locally constant on each stratum $\Acal_r$. Each point of $\Acal_r$ has a neighborhood basis whose members  meet $\Acal_g$ in a virtual classifying space of an  arithmetic group $P_g(r)$ defined below (for a more detailed discussion we refer to \cite{looijenga}, Example 3.5; see also Section 4 of \cite{chen-looijenga}), so that $R^\pt j_{g*}\QQ_{\Acal_g}$ can be identified with the rational cohomology of $P_g(r)$.

Let $H$ stand for  $\ZZ^2$ (with basis denoted $(e, e')$) and
endowed with the symplectic form characterized by  $\la e, e'\ra=1$. We also put $I:=\ZZ e$. We regard $H^g$  as a direct sum of symplectic lattices with $g$ summands.
In terms of the decomposition $H^g= H^r\oplus H^{g-r}$,  $P_g(r)$ is the group of symplectic transformations 
in $H^g$ that are the identity on $H^r\oplus 0$ and preserve $H^r\oplus I^{g-r}$. The orbifold fundamental group of $\Acal_r$ is isomorphic to 
$\Sp(H^r)$ (the isomorphism is of course given up to conjugacy) and its representation on a  stalk of $R^\pt j_{g*}\QQ_{\Acal_g}|\Acal_r$ corresponds  to its  obvious action (given by conjugation) on $P_g(r)$.  Note that this action is algebraic in the sense that
it extends to a representation of the underlying affine algebraic group (which assigns to a commutative ring  $R$ the group $\Sp (H^r\otimes R)$).  If $p\in \Acal_r$ and $U_p$ is a regular neighborhood of
$p$ in $\Acal^\bb_g$ such that the natural map $H^\pt(U_p\cap \Acal_g; \QQ)\to (R^\pt j_{g*}\QQ_{\Acal_g})_p$ is an isomorphism, then for every 
$r\le s\le g$  and $q\in U_p\cap \Acal_s$ the restriction map yields a map of $\QQ$-algebras  $(R^\pt j_{g*}\QQ_{\Acal_g})_p\to
(R^\pt j_{g*}\QQ_{\Acal_g})_q$.  Under the above identification this is represented  by the $\Sp (H^r)$-orbit of the 
obvious inclusion $P_g(s)\hookrightarrow P_g(r)$. Similarly,  the restriction to $\Acal_r\times\Acal_{r'}\subset \Acal_g\times\Acal_{g'}$ of the
natural sheaf homomorphism 
\[
R^\pt j_{g+g'{}*}\QQ_{\Acal_{g+g'}}|\Acal^\bb_g\times\Acal^\bb_{g'}\to R^\pt(j_g\times j_{g'})_*\QQ_{\Acal_g\times \Acal_{g'}} \cong 
R^\pt j_{g*}\QQ_{\Acal_g}\boxtimes R^\pt j _{g'{}*}\QQ_{\Acal_{g'}}
\]
(we invoked the  K\"unneth isomorphism) is  induced by the obvious embedding $P_g(r)\times  P_{g'}(r')\hookrightarrow P_{g+g'}(r+r')$, or rather its $\Sp (H^{r+r'})$-orbit. 

The proof of  the first assertion of our main theorem rests on careful study of  the Leray spectral sequence for the inclusion  $j_g: \Acal_g\subset \Acal_g^\bb$,
\begin{equation}\label{eqn:ss}
E^{p,q}_2=H^p(\Acal_g^\bb, R^qj_{g*}\QQ)\Rightarrow H^{p+q}(\Acal_g; \QQ).
\end{equation}
Such a spectral sequence can be set up in the category of mixed Hodge modules
(see \cite{saito:1}), so that  this is in fact  a spectral sequence of mixed Hodge structures.  

\begin{lemma}\label{lemma:stable1}
Let $r\le g$. Then the natural map $H^p(\Acal_g^\bb, R^\pt j_{g*}\QQ)\to H^p(\Acal_r^\bb, R^\pt j_{g*}\QQ)$ is an isomorphism for $p<r$ and is injective for $p=r$.
\end{lemma}
\begin{proof}
It suffices to show that  when $r<g$, the natural map $H^p(\Acal_{r+1}^\bb, R^\pt j_{g*}\QQ)\to H^p(\Acal_r^\bb, R^\pt j_{g*}\QQ)$ has this property.
For this we consider the exact sequence
\begin{multline*}
\cdots\to H^p_c(\Acal_{r+1},R^\pt j_{g*}\QQ)\to   H^p(\Acal_{r+1}^\bb, R^\pt j_{g*}\QQ)\to H^p(\Acal_r^\bb, R^\pt j_{g*}\QQ)\to\\\to H^{p+1}_c(\Acal_{r+1},R^\pt j_{g*}\QQ)\to\cdots
\end{multline*}
The restriction $R^q j_{g*}\QQ|\Acal_{r+1}$ is a local system whose monodromy  comes from an action of the algebraic group $\Spcal (H^r)$. Following the Borel-Serre result mentioned above, $H^i_c(\Acal_{r+1},R^\pt j_{g*}\QQ)$  vanishes for $i\le r$ and so the lemma follows.
\end{proof}

By viewing $I^{g-r}$ as the subquotient $(H^r\oplus I^{g-r})/(H^r\oplus 0)$ of $H^g$, we see that there is a natural homomorphism of arithmetic groups  $P_g(r)\to \GL (I^{g-r})=\GL (g-r,\ZZ)$.

\begin{lemma}\label{lemma:stable2}
The homomorphism  
$P_g(r)\to \GL (g-r,\ZZ)$  induces  an isomorphism on rational cohomology in degrees $<\half (g-r-1)$. In   
that range the rational cohomology of  $\GL (g-r,\ZZ)$ is stable and is canonically isomorphic to the cohomology of 
$\GL (\ZZ):=\cup_r\GL(r, \ZZ)$. The inclusion $P_g(r)\times P_{g'}(r')\subset P_{g+g'}(r+r')$  induces on rational cohomology in 
the stable range (relative to both factors) the coproduct  in the Hopf algebra $H^\pt (\GL (\ZZ); \QQ)$.
\end{lemma}
\begin{proof}
According to Borel \cite{borel:stable2}, the cohomology of the arithmetic group $\GL(r, \ZZ)$ with values in an irreducible representation of the underlying algebraic group
$\SLcal_r^{\pm}$ (the group of invertible matrices of determinant $\pm 1$) is zero in degree $<\half(r-1)$, unless the representation is trivial.
Let $N_g(r)$ be the kernel of $P_g(r)\to\GL (g-r, \ZZ)$. This is a nilpotent subgroup whose center, when written additively, may be identified with the symmetric quotient 
$\sym_2(I^{g-r})$ of $I^{g-r}\otimes I^{g-r}$. The quotient of $N_g(r)$ by this center is  abelian, and when written additively, naturally identified with the lattice $H^r\otimes I^{g-r}$. So in view of the Leray spectral sequence 
\[
H^p(\GL(g-r, \ZZ), H^q(N_g(r), \RR))\Rightarrow H^{p+q}(P_g(r), \RR),
\]
it suffices to show that $H^q(N_g(r); \RR)$  does not contain the trivial representation of $\SL^{\pm 1}(g-r, \RR)$  in  positive degree $q<\half(g-r-1)$. This follows from another Leray spectral sequence 
\[
H^s(I^{g-r}\otimes H^r, H^t(\sym_2 I^{g-r}, \RR)) \Rightarrow  H^{s+t}(N_g(r), \RR).
\]
The left hand side is isomorphic to  
\[
\wedge^s\Hom (I^{g-r}\otimes H^r, \RR)\otimes \wedge^t\Hom(\sym_2 I^{g-r},\RR)
\] 
as a representation  of $\SL^{\pm 1}(g-r,\RR)$. The invariant theory of $\SL (g-r; \RR)$ tells us that the trivial representations
in the tensor algebra generated by $\Hom (I^{g-r}, \RR)$  come from the formation of powers of the determinant 
$\wedge^{g-r}\Hom (I^{g-r}, \RR)\cong\RR$ (see for example \cite{fultonharris}). Since the displayed  representation  of
$\SL^{\pm 1}(g-r,\RR)$  is a quotient of this tensor algebra, it  will not contain the trivial representation when $0<s+2t<g-r$. Hence the first part lemma follows. The second assertion merely quotes a theorem of Borel \cite{borel:stable1} and the last assertion is easy.
\end{proof}

\begin{corollary}\label{cor:stable2}
For $q<\half(g-r-1)$, $R^q j_{g*}\QQ|\Acal^\bb_r$ is a constant local system whose stalk is canonically isomorphic to $H^q(\GL (\ZZ),\QQ)$. This identification is compatible with the multiplicative structure. It is also compatible with the coproduct in the sense that
when $0\le r'\le g'$, then in degree  $<\half\min\{g-r-1,g'-r'-1\}$,  the natural map $R^{\pt} j_{g+g'{}*}\QQ_{\Acal_{g+g'}}|\Acal^\bb_r\times\Acal^\bb_{r'}\to (R^\pt j_{g*}\QQ_{\Acal_g}|\Acal^\bb_r)\boxtimes (R^{\pt} j _{g'{}*}\QQ_{\Acal_{g'}}|\Acal^\bb_{r'})$ is stalkwise identified  with the coproduct on  $H^\pt(\GL (\ZZ); \QQ)$. 
\end{corollary}

\begin{proof}[Proof of the first assertion of Theorem \ref{thm:main}]
We have shown (Lemma \ref{lemma:stable1} and Corollary \ref{cor:stable2}) that  when $p<r$ and $q<\half(g-r-1)$ we have 
\[
E^{p,q}_2=H^p(\Acal_g^\bb, R^qj_{g*}\QQ)=H^p(\Acal^\bb_r, \QQ)\otimes H^q(\GL(\ZZ); \QQ)
\]
The Leray spectral sequences (\ref{eqn:ss}) for $j_{g*}$ and $j_{g+1 *}$ are compatible and so we may form a limit: we fix $p$ and $q$, but we let
$r$ and $g-r$ tend to infinity. This then yields a spectral sequence
\begin{equation}\label{eqn:ssinfty}
E^{p,q}_2=H^p(\Acal_\infty^\bb;\QQ)\otimes H^q(\GL(\ZZ); \QQ) \Rightarrow H^{p+q}(\Acal_\infty; \QQ).
\end{equation}
This spectral sequence is not just multiplicative,  but also compatible with the coproduct. So the differentials take primitive elements to primitive elements
(or zero) and the spectral sequence restricts to one of graded vector spaces by restricting to the primitive parts. The primitive part of $E^{p,q}_2$ is zero unless $p=0$ or $q=0$.  A theorem of Borel \cite{borel:stable1} tells us that $H^\pt(\GL(\ZZ); \QQ)_\pr$ has for every positive integer $r$ a generator $a_r$ in degree $4r+1$ (and is zero in all other positive degrees) and  that $H^\pt(\Acal_\infty; \QQ)_\pr$ has for every odd integer $s$ a primitive 
generator $ch_s$ in degree $2s$ (and is zero in all other positive degrees).  This implies that $d^k(1\otimes a_r)=0$ for $k=2,3 ,\dots 4r+1$, but that 
$y_r:= d^{4r+2}(1\otimes a_r)$ will be nonzero and primitive. We also see  that for $s>0$ odd, $H^{2s}(\Acal_g^\bb; \QQ)$ must contain a lift 
$\widetilde{ch}_s$ of $ch_s$. Since  $H^\pt(\Acal_\infty^\bb;\QQ)$ is a Hopf algebra, it then follows that the Hopf algebra $H^\pt(\Acal_\infty^\bb;\QQ)$ is primitively generated by  $y_1,y_2, \dots ,\widetilde{ch}_1,\widetilde{ch}_3, \widetilde{ch}_5, \dots$. So as a commutative $\QQ$-algebra it is freely generated by $y_1,y_2, \dots ,\tilde c_1,\tilde c_3, \tilde c_5, \dots$. 
\end{proof}

The spectral sequence (\ref{eqn:ssinfty}) suggests that the space $\Acal_\infty$ (which we did not define) has the homotopy type of a $B\!\GL(\ZZ)$-bundle over $\Acal^\bb_\infty$ (which we did not define either). Indeed, Charney and Lee provide in Thm.\ 3.2 of \cite{charneylee:satake} an appropriate homotopy substitute for such a fibration (which they attribute to Giffen), namely, a homotopy fibration whose fiber  is a model of $B\!\GL(\ZZ)^+$  (where `${}^+$' is the Quillen plus construction) and whose the total space is $\QQ$-homotopy equivalent to  $B\!\Sp(\ZZ)^+$,  so that the base
(which admits an explicit  description as the classifying space of a category) may be regarded as a  
$\QQ$-homotopy type representing  $\Acal^\bb_\infty$.

\begin{remark}\label{rem:}
The long exact sequence for the pair $(\Acal^\bb_g, \Acal_g)$ shows that the cohomology $H^\pt(\Acal^\bb_g, \Acal_g; \QQ)$ stabilizes as well with $g$
and is equal to the ideal in $\QQ[y_1,y_2, \dots ,\tilde c_1, \tilde c_3, \tilde c_5, \dots]$ generated by the $y_r$'s. We shall therefore denote this ideal by
$H^\pt(\Acal^\bb_\infty, \Acal_\infty; \QQ)$. We use the occasion to  point out that  the $y$-classes are canonically defined, but that this is not at all clear for the $\tilde c$-classes (for more on this, see Remark \ref{rem:tate2}).
\end{remark}

\begin{remark}\label{rem:classconstruction}
We can account geometrically  for the classes $y_r$  as follows. Denote the single point of $\Acal_0\subset \Acal_g^\bb$ by $\infty$ (the worst cusp of $\Acal_g^\bb$),  and take $g$ so large that the natural  maps  $H^{4r+1}(\GL (\ZZ); \QQ)\to H^{4r+1}(\GL (g, \ZZ); \QQ)\to (R^{4r+1} j_{g*}\QQ))_\infty$ and
$H^{4r+2}(\Acal_\infty^\bb,\Acal_\infty; \QQ)\to H^{4r+2}(\Acal_g^\bb,\Acal_g; \QQ)$ are isomorphisms.
Choose a regular neighborhood $U_\infty$ of $\infty$ in $\Acal^\bb_g$ so that if we put $\mathring{U}_\infty:=U_\infty\cap \Acal_g$, the natural maps 
\[
(R^{4r+1} j_{g*}\QQ)_\infty\leftarrow H^{4r+1}(\mathring{U}_\infty; \QQ)\xrightarrow{\delta} H^{4r+2}(U_\infty, \mathring{U}_\infty; \QQ)
\]
are also isomorphisms.  
If we identify  $a_r\in H^{4r+1}(\GL (\ZZ); \QQ)$ with its image in  $H^{4r+1}(\mathring{U}_\infty; \QQ)$, then  $\delta (a_r)\in H^{4r+2}(U_\infty,\mathring{U}_\infty; \QQ)$ is precisely the image  of $y_r$ under the restriction map 
$H^{4r+2}(\Acal_\infty ^\bb,\Acal_\infty; \QQ)\cong H^{4r+2}(\Acal_g^\bb,\Acal_g; \QQ) \to H^{4r+2}(U_\infty,\mathring{U}_\infty; \QQ)$.  

We may also get a homology class this way: the 
Hopf algebra  $H_\pt(\GL (\ZZ); \QQ)$ has a primitive generator in $H_{4r+1}(\GL (\ZZ); \QQ)$ that is dual to $a_r$ and if we represent 
this generator as $(4r+1)$-cycle $B_r$ in $\mathring{U}_\infty$, then $B_r$ bounds both in $U_\infty$ (almost canonically) and in $\Acal_g$ (not canonically). The two bounding $(4r+2)$-chains make up
a $(4r+2)$-cycle in $\Acal_g^\bb$ whose class $z_r\in H_{4r+2}(\Acal^\bb_g; \QQ)$ pairs nontrivially with the image of $y_r$ in 
$H^{4r+2}(\Acal^\bb_g; \QQ)$. 
\end{remark}

\section{The mixed Hodge structure on the primitive stable cohomology}

\begin{proof}[Proof that the $y$-classes  are of weight zero]
In view of Remark \ref{rem:classconstruction} it is enough to show that the image of $H^{\pt}(\GL (\ZZ); \QQ)$ in the stalk $(R^{\pt}j_{g*}\QQ)_\infty$ 
has weight zero.  For this we will need a toroidal resolution of $U_\infty$ as described in \cite{amrt}, but we will try to get by with the minimal 
input necessary (for a somewhat more  detailed review of this construction one may consult \cite{chen-looijenga}).  

Consider the symmetric quotient $\sym_2 \ZZ^g$ of $\ZZ^g\otimes \ZZ^g$ and regard it  as a lattice in the space $\sym_2 \RR^g$ of  quadratic forms on 
$\Hom(\ZZ^g, \RR)$.  The positive definite quadratic forms  make up a cone $C_g\subset \sym_2 \RR^g$ that is open and  convex and  is as 
such spanned by its intersection with $\sym_2 \ZZ^g$. Let $C_g^+\supset C_g$ be the convex cone spanned by 
$\overline C_g\cap \sym_2 \ZZ^g$; this  is just the set of  semipositive  quadratic forms on $\Hom(\ZZ^g, \RR)$ whose kernel is spanned  by 
its intersection with $\Hom(\ZZ^g, \ZZ)$. The obvious action of $\GL (g, \ZZ)$ on $\sym_2 \ZZ^g$ preserves both  cones and is proper on $C_g$. 

Consider the algebraic torus $T_g:=\CC^\times\otimes _\ZZ \sym_2 \ZZ^g$. If we apply the `log norm'   
$\lognrm: z\in \CC^\times\mapsto \log |z|\in \RR$ to the first tensor factor, we get a $\GL(g, \ZZ)$-equivariant homomorphism 
$\lognrm_{T_g}: T_g\to \sym_2 \RR^g$ with kernel the compact torus $U(1)\otimes _\ZZ \sym_2 \ZZ^g$ . We denote by $\Tcal_g\subset T_g$ the preimage of $C_g$ so that we have defined a proper
$\GL(g, \ZZ)$-equivariant homomorphism of semigroups $\lognrm_{\Tcal_g}: \Tcal_g\to C_g$. Since $\GL (g, \ZZ)$ acts properly on 
$C_g$ it does so 
on $\Tcal_g$ and hence the orbit space $\mathring{V}:=\GL (g, \ZZ)\bs \Tcal_g$  has the structure of a complex-analytic orbifold. 
There is a  
natural extension of $V\supset\mathring V$ in the complex analytic category  (it is in fact the Stein hull of $\mathring V$ in case $g>1$) that comes 
with a distinguished point that we will (for good reasons) also denote by $\infty$ and which is such  that $\mathring V$  is open-dense in $V$ and 
$(V,V\ssm\mathring V)$ is topologically the open cone over a pair of spaces  with vertex $\infty$. It has the property that
there exists an open embedding of  $U_\infty$ in $V$ that takes $\infty$ to $\infty$ and identifies  $U_\infty$ with a regular neighborhood of $\infty$ in $V$ in such a way that  $\mathring U_\infty=U_\infty\cap \mathring V$. This  justifies that we focus on the triple $(V, \mathring V;\infty)$. All else we need to know about $V$ is that the toroidal extension of $\mathring V$ that we  are about to consider provides a resolution of $V$ as an orbifold.

The universal cover of $\Tcal_g$ is contractible (with covering group $\sym_2 \ZZ^g$) and hence the 
universal cover of $\mathring V$ as an orbifold is also contractible and has  covering group $\GL (g, \ZZ)\ltimes \sym_2 \ZZ^g$ (it is in fact a virtual classifying space for this group). Similarly, the orbit space, $\Ical_g:=\GL(g, \ZZ)\bs C_g$ exists as a real-analytic orbifold and is a virtual classifying space for $\GL(g, \ZZ)$. The map
$\lognrm_{T_g}$ induces a projection $\nu: \mathring V\to \Ical_g$ and the classes that concern us lie in the image of
\begin{equation}\label{eqn:classmap}
H^\pt(\GL(\ZZ); \QQ)\to H^\pt(\GL(g, \ZZ); \QQ)\to H^\pt (\Ical_g; \QQ)\xrightarrow{\nu^*} H^\pt (\mathring V; \QQ).
\end{equation}

A \emph{nonsingular admissible decomposition} of $C_g^+$ is a collection $\{\sigma\}_{\sigma\in\Sigma}$ of closed cones in $C_g^+$, each of which
is spanned by a partial basis of $\sym_2 \ZZ^g$, such that the collection is closed under `taking faces' and `taking intersections' and  whose relative interiors are pairwise disjoint with union $C_g^+$.  Let $\Sigma$  be such a decomposition  that is also $\GL(g, \ZZ)$-invariant and is fine enough in the sense that every $\GL (g, \ZZ)$-orbit in $C^g_+$ meets every member of $\Sigma$ in at most one point. Such decompositions exist \cite{amrt}. (One usually also requires that $\GL (g, \ZZ)$ has only finitely many orbits in 
$\Sigma$, but this is in fact implied by the other conditions, see \cite{looijenga:cone}.)
The associated torus embedding $T_g\subset T_g^\Sigma$ is then nonsingular and comes with an action of $\GL(g, \ZZ)$. We denote by $\Tcal_g^\Sigma$ the interior of the closure of $\Tcal_g$ in $T_g^\Sigma$. This is an open $\GL(g, \ZZ)$-invariant  subset of $ T_g^\Sigma$ on which $\GL(g, \ZZ)$ acts properly 
so that $V^\Sigma:=\GL(g, \ZZ)\bs \Tcal_g^\Sigma$ exists as an analytic orbifold. It is of the type alluded to above: we have a natural proper morphism 
$f: V^\Sigma\to V$ that is complex-algebraic over $V$ and is an isomorphism over $\mathring V$. Moreover, the exceptional set is a simple normal crossing divisor in the orbifold sense. 

As for every  torus embedding there is also a real counterpart in the sense that $\lognrm_{\Tcal_g}$ extends in a $\GL(g, \ZZ)$-equivariant manner to a proper and surjective map $\lognrm_{\Tcal_g^\Sigma}: \Tcal_g^\Sigma\to C_g^\Sigma$, where $C_g^\Sigma$ is a certain stratified locally compact Hausdorff space which contains $C_g$ as an open dense subset. In the present case $C_g^\Sigma$ is simply a manifold with corners, because $\Sigma$ is nonsingular. 
The strata of $C_g^\Sigma$ are indexed by $\Sigma$, with the stratum defined by $\sigma$ being the image of $C_g$ under the projection along the real subspace of $\sym_2 \RR^g$ spanned by $\sigma$. So each stratum of $C_g^\Sigma$ appears as a convex open subset of some vector space and it is all of this vector space precisely when the relative interior of $\sigma$ is contained $C_g$. This is also equivalent to the stratum having compact closure in $C_g^\Sigma$. 

Let us call a \emph{wall} of $C_g^\Sigma$, the closure of a stratum defined by a ray ($=$ a one-dimensional member) of $\Sigma$. So a wall is compact if and only if the
associated ray lies in $C_g\cup\{ 0\}$. We denote by $\partial_\pr C_g^\Sigma$ the union of these compact walls. This is a closed subset of 
$C_g^\Sigma$ and its covering by such compact walls is a \emph{Leray covering}: the covering is locally finite and each nonempty intersection is contractible (and is in fact the closure of a  stratum). Its nerve is easily  expressed  in terms of $\Sigma$. Let us say that a member of $\Sigma$ is \emph{proper} if it is contained in $C_g\cup\{ 0\}$.  The proper members of $\Sigma$  make up a subset $\Sigma_\pr\subset \Sigma$ that is also closed under `taking faces' and `taking intersections' and their union makes up a $\GL (g, \ZZ)$-invariant cone contained in $C_g\cup \{0\}$. If we projectivize that cone we get a simplicial complex in the real projective space of $\sym_2 \RR^g$ that we  denote by $P(\Sigma_\pr)$. A vertex of $P(\Sigma_\pr)$  corresponds of course to a ray of $\Sigma_\pr$ and this in turn defines a compact wall of $C_g^\Sigma$. In this way $P(\Sigma_\pr)$  can be identified in a $\GL(g, \ZZ)$-equivariant manner with the nerve complex of the covering of $\partial_\pr C_g^\Sigma$ by the compact walls of $C_g^\Sigma$.  A standard argument shows that we have a $\GL(g, \ZZ)$-equivariant homotopy equivalence between $\partial_\pr C_g^\Sigma$ and the nerve $P(\Sigma_\pr)$ of this covering. 

Each stratum closure in $C_g^\Sigma$ can be retracted in a canonical manner onto its intersection with $\partial_\pr C_g^\Sigma$ and we thus 
find a $\GL(g, \ZZ)$-equivariant deformation retraction $C_g^\Sigma \to \partial_\pr C_g^\Sigma$. This shows at the same time that the 
inclusion  $C_g\subset C_g^\Sigma$ is a $\GL(g, \ZZ)$-equivariant homotopy equivalence. So if we put $\Ical_g^\Sigma:=\GL(g, \ZZ)\bs C_g^\Sigma$ and
$\partial_\pr\Ical_g^\Sigma:=\GL(g, \ZZ)\bs\partial_\pr C_g^\Sigma$, then we end up with  homotopy equivalences
$\Ical_g\subset \Ical_g^\Sigma\supset \partial_\pr\Ical_g^\Sigma$. We also have a homotopy equivalence $\partial_\pr\Ical_g^\Sigma\sim \GL(g, \ZZ)\bs P(\Sigma_\pr)$.

Taking the preimage under $\lognrm$ makes walls of $C_g^\Sigma$ correspond to irreducible components of the toric boundary 
$\Tcal_g^\Sigma\ssm\Tcal_g$ and a wall of $C_g^\Sigma$ is compact if and only if the associated  irreducible component is.
So the preimage $\partial_\pr \Tcal_g^\Sigma$ of $\partial_\pr C_g^\Sigma$ is the union of the compact irreducible components 
of the toric boundary. It is clear that  $P(\Sigma_\pr)$ is also the nerve of the  covering  of $\partial_\pr \Tcal_g^\Sigma$ by its irreducible components.
The image of $\partial_\pr \Tcal_g^\Sigma$ in $V$ (in other words, its $\GL(g, \ZZ)$-orbit space) is the normal crossing divisor $f^{-1}(\infty)$.
The inclusion $f^{-1}(\infty)\subset V^\Sigma$  is also a deformation retract. So in the commutative diagram
\[
\begin{tikzcd}
\mathring{V}\arrow[two heads]{d} \arrow[hook]{r}
& V^\Sigma \arrow[two heads]{d}\arrow[hookleftarrow]{r}  
& f^{-1}(\infty) \arrow[two heads]{d}\\
\Ical_g \arrow[hook]{r} & \Ical^\Sigma_g \arrow[hookleftarrow]{r}  & \partial_\pr \Ical ^\Sigma
\end{tikzcd}
\]
the inclusion on the top right and those at the bottom are homotopy equivalences. It follows that the composite map in diagram (\ref{eqn:classmap}) factors through the rational cohomology of 
$\Ical_g^\Sigma$ and hence also through the rational cohomology of $V^\Sigma$ and that the nonzero classes in  
$H^\pt(V^\Sigma; \QQ)\cong H^\pt(f^{-1}(\infty); \QQ)$ that we thus obtain come from
the nerve of  the covering of $f^{-1}(\infty)$ by its irreducible components. Such classes are known to be of weight zero \cite{deligne:hodge3}.
\end{proof}

\begin{remark}\label{rem:tate2}
Goresky and Pardon \cite{goresky-pardon} have constructed a lift $c_r^\bb$ of the \emph{real} Chern class  $c_r\in H^{2r}(\Acal_g; \RR)$ to $H^{2r}(\Acal^\bb_g; \RR)$. The second author \cite{looijenga} recently proved that $c_r^\bb$ (and hence also the corresponding Chern character $ch_r^\bb$) lies in  $F^rH^{2r}(\Acal^\bb_g; \RR)$.  So the class of the Tate extension in Remark \ref{rem:tate1}  is up to a rational number given by  the value of $c^\bb_{2r+1}$  on the class $z_r\in H_{4r+2}(\Acal_g^\bb; \QQ)$ found in Remark \ref{rem:classconstruction} (two choices of $z_r$ differ by a class of the form $j_{g *}(w)$ with $w\in H_{4r+2}(\Acal_g; \QQ)$ and $c^\bb_{2r+1}$ takes on such a class the rational value $c_{2r+1}(w)$). Arvind Nair, after learning of our theorem, informed us that his techniques enable him to show that this extension class is nonzero. Subsequently a different proof (based on the Beilinson regulator)  was given in  \cite{looijenga}.
\end{remark}

\end{document}